\documentclass[12pt,a4 paper]{article}
\usepackage{amssymb}
\usepackage{amsmath}
\usepackage{amsthm}
\usepackage{amsfonts}
\usepackage{amstext}
\usepackage{cite}

\newtheorem{Theorem}{Theorem}[section]
\newtheorem{Lemma}[Theorem]{Lemma}
\newtheorem{Proposition}[Theorem]{Proposition}
\newtheorem{Definition}[Theorem]{Definition}
\newtheorem{Corollary}[Theorem]{Corollary}

\newtheorem{Remark}[Theorem]{Remark}
\newcommand{\dif}{\mathrm{d}}
\newcommand{\di}{\mathrm{div}}
\newcommand{\p}{\partial}

\def \inner#1#2{\langle\,#1, #2\,\rangle}

\def \RR{\mathbb{R}}

\begin{document}

\title{Existence, regularity and uniqueness of weak solutions for  a class of incompressible generalized Navier-Stokes system with slip boundary conditions in $\RR^3_+$ \footnote{This research is partially supported by NSFC(11201411) and  Natural Science Funds of Jiangxi Science and Technology (20122BAB211004) }}
\author{Aibin Zang}

\date{}

\maketitle

\begin{center}
The School of Mathematics and Computer Science, Yichun University, Yichun, Jiangxi, P.R.China, 336000\\

\footnotesize{Email: zangab05@126.com}

\end{center}

\begin{abstract}
We obtain the existence, regularity, uniqueness of  the non-stationary problems of a class of non-Newtonian fluid  is a power law fluid with $p>\frac{9}{5}$ in the half-space under slip boundary conditions.\\
\textbf{Keywords}: Non-Newtonian Fluid; Navier's slip boundary conditions; weak solution;\\
\textbf{Mathematics Subject Classification(2000)}: 76D05, 35D05,54B15, 34A34.
\end{abstract}

\section{Introduction}

Let $\Omega\subset\mathbb{R}^3$ be an open set.  For any $T<\infty$, set $Q_T=\Omega\times (0,T).$ The motion of a homogeneous, incompressible fluid through $\Omega$ is governed  by the following equations 
   \begin{equation}\label{1.1}\tag{1.1}
\left\{\begin{aligned} &\partial_t u-\di S+(V\cdot\nabla)u+\nabla \pi=f, &\mbox{in}~~Q_T \\[3mm]
&\nabla\cdot u = 0, &\mbox{in}~~Q_T \\[2mm]
&u|_{t=0}=u_0(x),& \mbox{in}~~\Omega,
\end{aligned}\right.
\end{equation}
 where $u$ is the velocity,  $\pi$ is pressure and $f$ is the force, $V$ is chosen a solenoid vector function and  tangential to the boundary of $\Omega$, $u_0$ is initial velocity and $S=(s_{ij})_{i,j=1}^n$ is stress tensor. The above system (1.1) has to be completed by boundary conditions except that $\Omega$ is the whole space and by constitutive assumptions for the extra tensor. Concerning the former  we can impose the following Navier slip boundary conditions 
 \begin{equation}\label{1.2} \tag{1.2}
 \begin{aligned}
 u\cdot n=0,  (S\cdot n)_\tau-\alpha u_\tau=0,\mbox{on}~~\p\Omega\times(0,T),
 \end{aligned}
 \end{equation}
 where $\alpha$ is the frictional constant.

Many extra tensors are characterized by Stoke's law $S=\nu D(u)$,  where $D(u)$ is the symmetric velocity  gradient, i.e. $$D_{ij}(u)=\frac{1}{2}\left(\frac{\p u_i}{\p x_j}+\frac{\p u_j}{\p x_i}\right).$$ Assume that $\nu$ is a constant and $V=u$, (1.1) is called incompressible Navier-Stokes equations.

However, there are phenomena that can  be described by  $\nu=\nu(|D(u)|)$ with power-law ansatz to model certain non-Newtonian behavior of the fluid flows, and they are frequently used engineering literature. We can refer the book by Bird, Armstrong and Hassager \cite{BAH} and the survey paper due to M\'{a}lek and Rajagopal \cite{MAR}.  Typical  examples for this constitutive relations are 
\begin{equation}\label{1.3}\tag{1.3}
\begin{aligned}
S(D(u))=\mu(\delta +|D(u)|)^{q-2}D(u)\\
S(D(u))=\mu(\delta +|D(u)|^2)^{\frac{q-2}{2}}D(u),
\end{aligned}
\end{equation}
with $1<q<\infty, \delta\ge 0,$ and $\mu>0$.

The mathematical analysis of these models  started with the work of Lady\v{z}henska \cite{LA01},\cite{LA02},\cite{LA03}. She investigated the well-posedness of the initial boundary value problem with non-slip boundary conditions, associated with the stress tensor (1.3).  In 1969, J.L. Lions \cite{Lions} proved some existence results for $p-$Laplacian equation with $p\ge 1+\frac{2n}{n+2}$ and the uniqueness for $p\ge\frac{n+2}{n}$ under no-slip boundary conditions.  In those papers, the authors applied the properties of  monotone operator and Minty trick theory for the stress tensor satisfies the strict monotonicity and coercivity.

Over these years, Lady\v{z}henska's and Lions' work were improved   in several directions by different authors.  In particular, for the steady problem, there are several results proving existence of weak solution in bounded domain \cite{DMM,fr,FR01}, interior regularity \cite{AMG,NW} and very recently regularity up to boundary for the Dirichlet problem \cite{b1,BE01,BE02,BE03,BE04,BE05,BE06,CR01, CR02,SH}. Concerning  the time-evolution Dirichlet problem in a 3D domain, J. M\'{a}lek, J. Ne\u cas, and M. R\r u\u zi\u cka\cite{MNR} study the weak solution for $p\ge 2$. Later, L. Diening et.al have  recent advances on the existence of weak solutions in \cite{DMM} for $p>\frac{8}{5}$ and in \cite{DRW} for $p>\frac{6}{5}.$ There are also many papers dealing with regularity of for evolution Dirichlet boundary problems and we refer instance to \cite{AMANN,BA,BP,BEKP,BE03,BE04,BE05,BE06}. In the three-dimensional cube with space periodic boundary conditions, there are a lot of literatures for the well-posedness of this model, we refer to the monograph \cite {ma} and papers\cite{BDR,dien}.

It should be emphasized that theoretical contributions mostly concern the homogenous boundary condition and space periodic boundary conditions. However, many other boundary conditions are important for engineer experiment and computation science. Commonly used boundary conditions are Navier-type boundary conditions, which were introduced by Navier in \cite{NA}.  Newtonian fluid under Navier slip boundary conditions was studied by many mathematician,\cite{B2, B3} and \cite{xin}. However,  there are not too many results for non-Newtonian fluid.  In \cite{b1, EM}, the authors investigated the regularity of steady flows with  shear-dependent viscosity on the slip boundary conditions. M. Bul\'{\i}\v{c}ek, J.M\'{a}lek and K.R. Rajagopal \cite{bu} obtained the weak solution for the evolutionary generalized Navier-stokes-like system of pressure and shear-dependent viscosity on the Navier-type slip boundary conditions in the bounded domain.

   In this paper, we consider the problem  (1.1) with stress tensor $S$ induced by $p-$potential as in Definition 2.1, when $\Omega=\mathbb{R}^3_+$,  under  the following slip boundary conditions
 \begin{equation}\label{1.4}\tag{1.4}
 u\cdot n|_{x_3=0}=0, \,\,~((S(D(u))\cdot n)-(n\cdot S(D(u))\cdot n)n)|_{x_3=0}=0.
 \end{equation}
In fact, this problem corresponds to the free boundary problem for the non-Newtonian fluids with free surface supposed invariable.

 Since  we choose the stress tensor induced by a $p-$potential, and then we will obtain  the equivalent conditions:
$$u_3|_{x_3=0}=\left.\frac{\p u_i}{\p x_3}\right|_{x_3=0}=0\,\,(i=1,2).\eqno{(1.5)}.$$ From these conditions, we extend to the external force term $f$ and initial velocity $u_0$ to whole space by mirror reflection method and change (1.1) into a Cauchy problem. Hence, we can focus on  the regularity estimates, uniqueness and existence  of this Cauchy problem.  Then one can obtain the existence of the solution by Galerkin Method in the half space.

The paper is organized as follows. In section 2, after recalling the notation and presenting some preliminary  results, we give the definitions of the $p-$potential and weak solutions.  We also present the existence of the divergence-free base with boundary conditions (1.5) in $W^{2,2}$. In section 3, we show some theorems for the existence, regularity and uniqueness of weak solutions for the system (1.1) with boundary conditions (1.5). 

\section{Preliminaries}
In this section, we will give some assumptions, function spaces and definitions for weak solution. We will show the Korn-type inequalities for unbounded domain and construction of the basis with boundary conditions (1.6).
Let $M^{n\times n}$ be the vector space of all symmetric $n\times n$ matrices $\xi=(\xi_{ij})$. We equip $M^{n\times n}$ with scalar product $\xi:\eta=\sum_{i,j=1}^n \xi_{ij}\eta_{ij}$ and norm $|\xi|=(\xi:\eta)^{\frac{1}{2}}.$
\begin{Definition}
Let $p>1$ and let $F: \mathbb{R}_+\bigcup\{0\}\to\mathbb{R}_+\bigcup\{0\}$ be a convex function, which is $C^2$ on the $\mathbb{R}_+\bigcup\{0\}$, such that $F(0)=0,\,\ F'(0)=0.$ Assume that the induced function $\Phi:M^{n\times n}\to\mathbb{R}_+\bigcup\{0\}$, defined through $\Phi(B)=F(|B|),$ satisfies
\begin{alignat}{12}
&\sum_{jklm}(\p_{jk}\p_{lm}\Phi)(B)C_{jk}C_{lm}\ge\gamma_1(1+|B|^2)^{\frac{p-2}{2}}|C|^2,\tag{2.1}\\
&|(\nabla_{n\times n}^2\Phi)(B)|\le\gamma_2(1+|B|^2)^{\frac{p-2}{2}}\tag{2.2}
\end{alignat}
for all $B,C\in M^{n\times n}$ with constants $\gamma_1,\gamma_2>0.$ Such a function $F$, resp. $\Phi$, is called a $p-$\textbf{potential}.
\end{Definition}

We define the extra stress $S$ induced by $F$, resp. $\Phi,$ by
$$S(B)=\nabla_{n\times n}^2\Phi(B)=F'(|B|)\frac{B}{|B|}$$ for all $B\in M^{n\times n}\setminus\{ \mathbf{0}\}.$
From (2.1), (2.2) and $F'(0)=0$, it easy to know that $S$ can be continuously extended by $S(\mathbf{0})=\mathbf{0}$.

As in the\cite{dien} and\cite{ma}, one can obtain from (2.1) and (2.2) the following properties of $S$.
\begin{Theorem}
There exist constants $c_1,c_2>0$ independent of $\gamma_1,\gamma_2$ such that for all $B,C\in M^{n\times n}$ there holds
\begin{alignat*}{12}
&S(\mathbf{0})=\mathbf{0},\,\ \tag{2.3}\\
&\sum_{i,j}(S_{ij}(B)-S_{ij}(C))(B_{ij}-C_{ij})\ge c_1\gamma_1(1+|B|^2+|C|^2)^{\frac{p-2}{2}}|B-C|^2,\,\, \\
&\sum_{i,j}S_{ij}(B)B_{ij}\ge c_1\gamma_1(1+|B|^2)^{\frac{p-2}{2}}|B|^2,\tag{2.4}\\
&|S(B)-S(C)|\le c_2\gamma_2(1+|B|^2+|C|^ 2)^{\frac{p-2}{2}}|B-C|,\\
&|S(B)|\le c_2\gamma_2(1+|B|^2)^{\frac{p-2}{2}}|B|.\tag{2.5}
\end{alignat*}
\end{Theorem}
In the following part of this section, we will give some function spaces and the definition of weak solutions for the system (1.1).
\begin{alignat*}{12}
&D(\mathbb{R}^3_+)=\{u\in C^\infty_0(\overline{\mathbb{R}^3_+}): u_3=0\,\, \text{on}\,\, x_3=0\},V_p(\mathbb{R}^3_+)=\overline{\{u\in D(\mathbb{R}^3_+):\nabla\cdot u=0\}}^{\|\nabla\cdot\|_{L^p}},\,\
\\
& H=\overline{D(\mathbb{R}^3_+)}^{\|\cdot\|_{L^2}}.
\end{alignat*}
Denote $\Omega_R=\{x\in\mathbb{R}^3_+: |x|\le R\}$ for $R>0$ , then we have the corresponding spaces for domain $\Omega_R$  as follows,
\begin{alignat*}{12}
&D(\Omega_R)=\{u\in C^\infty_0(\overline{\Omega}_R): u_3=0\,\, \text{on}\,\, x_3=0\},\,\,V_p(\Omega_R)=\overline{\{u\in D(\Omega_R):\nabla\cdot u=0\}}^{\|\nabla\cdot\|_{L^p}},
\\
&  H(\Omega_R)=\overline{D(\Omega_R)}^{\|\cdot\|_{L^2}}.\end{alignat*}
Let $\Gamma_R^1=B_R\cap\{x_3=0\},\,\, \Gamma_R^2=\p B_R\cap\mathbb{R}^3_+$ and $Q=\mathbb{R}^3_+\times [0,T],\,Q_T^R=\Omega_R\times[0,T].$
\begin{Definition}

Let $\frac{6}{5}\le p<\infty,$ under the assumption of Definition  2.1. Let  $f\in H $ or $f\in V_p^*$, which is the dual space of $V_p$, and $u_0\in H$ with $\nabla\cdot u_0=0$ in the sense of distribute. A vector function $u\in L^\infty(0,T;H)\cap L^p(0,T;V_p)$ is called a weak solution to (1.1) if the following identity
\begin{equation}
\begin{aligned}
&-\int_{Q_T} (u\cdot\p_t\phi)\dif x\dif t+&\int_{Q_T}(S(x,t,D(u))-V\otimes u):D(\phi)\dif x\dif t\\
&&=\int_{Q_T} f\cdot\phi \dif x\dif t+\int_{\mathbb{R}^3_+}u_0\cdot\phi(0)\dif x
\end{aligned}\tag{2.6}
\end{equation}
holds for all $\phi\in C^{\infty}(\overline{\mathbb{R}^3_+}\times [0,T])$ with $\di\phi=0,\,\, \phi_3|_{x_3=0}=0$, and $\text{supp} \phi\subset \overline{\mathbb{R}^3_+}\times [0,T).$
\end{Definition}

We need to recall the following Korn-type inequality (see Theorem 3-2 in \cite{ol}.)
\begin{Theorem}
Let $K$ be cone in $\RR^n$ and $p>1$. If $\int_K |D(u)|^p\dif x<+\infty$, then there is a skew-symmetric matrix $A$ with constant coefficients such that 
\begin{equation*}
\begin{aligned}
\int_K|\nabla (u(x)-Ax)|^p\dif x\le C\int_K|D(u)|^p\dif x
\end{aligned}
\end{equation*}
where the constant $C$ does not depend on $u$.
\end{Theorem}

The previous result leads to the following.
\begin{Corollary}
There exists a constant C depending only on  $p$ such that $$\|\nabla u\|_p\le C\|D(u)\|_p$$ for all $u\in C_0^{\infty}(\overline{\mathbb{R}^3_+})$.
\end{Corollary}
\begin{proof}
Since the domain $\RR^3_+$ is a special cone in $\RR^3$, therefore, along the proof Corollary 1 in \cite{ga},  it is easy to get the result by Theorem 2.5.\end{proof}

 To construct the basis in $W^{2,2}(\Omega_R)$  with the boundary conditions (1.5), we consider the following  problem
 \begin{equation}\label{2.7}\tag{2.7}
\left\{\begin{aligned} &-\Delta u+\nabla p=f, &&\mbox{in}~~\Omega_R \\[3mm]
&\nabla\cdot u = 0, &&\mbox{in}~~\Omega_R, \\[2mm]
&u_3=0,~~~\frac{\p u_1}{\p x_3}=\frac{\p u_2}{\p x_3}=0  &&\mbox{on}~~\Gamma_R^1,\\[2mm]\\
&u=0    &&\mbox{on}~~\Gamma_R^2,\\
\end{aligned}\right.
\end{equation}
The following definition 2.6, Lemma 2.7 and its proof will be found in \cite{mar}.
\begin{Definition}
By a weak solution of the problem (2.7) we mean a function $u(x)\in V_2(\Omega_R)$ such that $$(D(u),D(v))=(f,v), \,\,\forall v\in V_2(\Omega_R)$$
\end{Definition}
\begin{Lemma}
Assume $f\in H(\Omega_R),$ then there exists a unique solution $(u(x),p(x))$ to problem (2.7) such that $u\in V_2(\Omega_R)\cap W^{2,2}(\Omega_R)$. Moreover, the following estimates hold:
\begin{equation}\label{2.8}\tag{2.8}
\begin{aligned}
&\|D(u)\|_{L^2(\Omega_R)}\le C\|f\|_{H(\Omega_R)}\\
&\|\nabla^2 u\|_{L^2(\Omega_R)}+\|\nabla p\|_{L^2(\Omega_R)}\le C\left(\|f\|_{H(\Omega_R)}+\|u\|_{V_2(\Omega_R)}\right)\\
&\|\nabla u\|_{L^3(\Omega_R)}\le C\left(\|f\|_{H(\Omega_R)}^{\frac{1}{2}}\|\nabla u\|^{\frac{1}{2}}_{L^2(\Omega_R)}+\|u\|_{V_2(\Omega_R)}\right),
\end{aligned}
\end{equation}
where $C$ is independent of $u,\,f$.
\end{Lemma}

With the aid of previous lemma, one can prove the following proposition
\begin{Proposition} The eigenvalue problem
\begin{equation*}
\left\{\begin{aligned} &-\Delta u+\nabla p=\lambda u, &&\mbox{in}~~\Omega_R \\[3mm]
&\nabla\cdot u = 0, &&\mbox{in}~~\Omega_R, \\[2mm]
&u_3=0,~~~\frac{\p u_1}{\p x_3}=\frac{\p u_2}{\p x_3}=0  &&\mbox{on}~~\Gamma_R^1,\\[2mm]\\
&u=0    &&\mbox{on}~~\Gamma_R^2,\\
\end{aligned}\right.
\end{equation*}
$\lambda \in\mathbb{R},\,\, u\in V_2(\Omega_R) $ admits a denumberable positive eigenvalue $\{\lambda_i\}$ clustering at infinity. Moreover, the corresponding eigenfunctions $\{a_i\}$ are in $W^{2,2}(\Omega_R),$ and associate pressure fields $p_i\in W^{1,2}(\Omega_R)$. Finally, $\{a_i\}$ are orthogonal and complete in $H(\Omega_R)$ and $V_2(\Omega_R).$
\end{Proposition}
\begin{proof}
The mapping $A:\ \ f\longrightarrow u$ defined by Lemma 2.7 is  linear and continuous from $H(\Omega_R)$ onto $V_2(\Omega_R)$, into $W^{1,2}(\Omega_R)$. Since $\Omega_R$ is bounded, by Rellich  Theorem, we know that $W^{1,2}(\Omega_R)\hookrightarrow L^2(\Omega_R)$ is compact. It is easy to know that operator $A$ is a positive symmetric and self-adjoint operator on $L^(\Omega_R)$. Therefore, $A$ possess an sequence of eigenfunctions $a_i$:
\begin{eqnarray*}
&&A a_i=\lambda_i a_i\,\, k\ge 0, \lambda_i>0,\,\, \lambda_i\to\infty ~\mbox{as}~ k\to\infty\\
&&(a_i,a_j)_{L^2}=\delta_{i,j},\,\,(D(a_i),D(a_j))_{L^2}=\lambda_k\delta_{i,j}.
\end{eqnarray*}
By Lemma 2.7, we can get for each $i$, there exists $p_i$ with the estimates (2.8).
\end{proof}

\section{Main results and their proofs  }
To study the well-posedness  of problem (1.1), we
define a reflection as follows
\begin{equation}\label{3.1}\tag{3.1}
u^*(x)=\left\{\begin{aligned} &(u_1(x_1,x_2,x_3),u_2(x_1,x_2,x_3),u_3(x_1,x_2,x_3)) &&\mbox{if\,\,}x_3\ge0; \\[3mm]
&(u_1(x_1,x_2,-x_3),u_2(x_1,x_2,-x_3),-u_3(x_1,x_2,-x_3), &&\mbox{if\,\,}x_3<0.
\end{aligned}\right.
\end{equation}

 Next, we will show the existence,  uniqueness  of strong solutions to the problem (1.1). We give the definition of the strong solution for the problem (1.1) as follows
\begin{Definition}
We say a couple $(u,\pi)$ is a strong solution to problem (1.1) if
\begin{equation}\label{3.2}\tag{3.2}
\begin{aligned}
&u\in L^\infty(0,T;W_{\rm loc}^{1,2}(\overline{\Omega}))\cap L^{p}(0,T;W^{2,p}_{\rm loc}(\overline{\Omega}))\cap L^p(0,T; V_p)\cap L^\infty (0,T;H)\\
&\frac{\p u}{\p t}\in L^2(0,T;L^2_{\rm loc}(\overline{\Omega})); \pi\in L^{p'}(0,T;L^{p'}_{\rm loc}(\overline{\Omega}).
\end{aligned}
\end{equation}
where $p'=\frac{p}{p-1}$ and satisfies the weak formulation
\begin{equation}\label{3.3}\tag{3.3}
\begin{aligned}
\int_{\Omega}\frac{\p u}{\p t}\varphi\dif x +\int_{\Omega} S(D(u)):D(\varphi)\dif x +\int_{\Omega} (V\cdot\nabla)u\cdot\varphi\dif x\\
=\int_{\Omega} \pi\di \varphi\dif x+\int_{\Omega} f \varphi\dif x.
\end{aligned}
\end{equation}
 holds for all $\varphi\in C^{\infty}_0(\overline{\Omega})$ and almost all $t\in (0,T)$, at same time, the boundary conditions hold in the sense of trace.
\end{Definition}

At first, we  provide the definition of difference  and recall a well-known result. Fixed any  domain $\Omega\subset\overline{\mathbb{R}^3_+}, \Omega'\subset\subset\Omega$, and we put $\delta(\Omega',\Omega)=\rm{dist}(\Omega',\p\Omega\setminus\{x_3=0\}).$

\begin{Definition}
For  $g:\Omega\longrightarrow\mathbb{R}^3$ we set
$$(\Delta_{\lambda,k}g)(x)=g(x+\lambda e_k)-g(x),\,\,x\in\Omega',0<\lambda<\delta(\Omega',\Omega),k=1\cdots 3.$$
where $e_1,e_2,e_3$ is the canonical base of $\mathbb{R}^3.$ We shall omit the dependence on $k$ where the meaning is clear.

\end{Definition}
\begin{Lemma}
For any $u\in W^{1,p}(\Omega)$ and $0<|\lambda|<\delta(\Omega',\Omega)$ it holds
$$\|\Delta_{\lambda,k} u\|_{p,\Omega'}\le|\lambda|\|u_{,k}\|_{p,\Omega}.$$
\end{Lemma}

For above $\Omega, \Omega' $ and $\delta(\Omega',\Omega)$, then the following theorem and lemma show that the regularity and uniqueness  of the weak solutions to the problem (1.1).
\begin{Theorem}
Let $\frac{9}{5}<p<2, f\in L^{p'}(Q_T)$, $V\in L^\infty(0,T;W^{2,2}(\Omega))$, $u_0\in V_2\cap H$ satisfy the boundary conditions (1.5), and $S$ be  given by a p-potential from Definition 2.1. If  $u\in L^p(0,T; V_p)\cap L^\infty (0,T;H)$ is the weak solution for problem (1.1), then this solution is also a unique strong solution to problem (1.1) such that
\begin{alignat*}{12}
&\|u\|_{L^{\infty}(0,T; W^{1,2}(\Omega'))\cap L^p(0,T; W^{2,p}(\Omega'))}\le C(|\Omega'|,u_0,f,T,\|V\|_{L^\infty((0,T)\times\Omega)}),\,\ \tag{3.4}\\
&\int^T_0\int_{\Omega'}(1+|D(u)|)^{p-2}|\nabla D(u)|^2\dif x\dif t\le C(|\Omega'|,u_0,f,T,\|V\|_{L^\infty((0,T)\times\Omega)}).\tag{3.5}\\
&\|\frac{\p u}{\p t}\|^2_{L^2((0,T)\times\Omega')}+\|\Phi(D(u))\|_{L^\infty(0,T;L^1(\Omega'))}\le C(\delta(\Omega',\Omega),|\Omega'|,u_0,f,T,\|V\|_{L^\infty((0,T)\times\Omega)}).\tag{3.6}
\end{alignat*}
\end{Theorem}
\begin{proof}
  Firstly  we extent the $u_0$ and force term $f$ to the whole space by the reflection defined in (3.1), however, since $V\in W^{2,2}$, we need apply the extension of Theorem 5.19 in \cite {ad}, Denote these functions by $u^*_0,\,\, \,\,  f^*,\,\,\,\,V^*$ respectively.

 We begin to consider the Cauchy problem as follows
\begin{equation}\label{3.7}\tag{3.7}
\left\{\begin{aligned} &\partial_t v-\di S(D(v))+(V^*\cdot\nabla)v+\nabla \pi=f^*, &&\mbox{in}~~\mathbb{R}^3\times (0, T) \\[3mm]
&\nabla\cdot v = 0, &&\mbox{in}~~\mathbb{R}^3\times (0, T), \\[2mm]
&v|_{t=0}=u_0^*(x),&& \mbox{in}~~\mathbb{R}^3.
\end{aligned}\right.
\end{equation}
From \cite {po}, There exists a weak solution $u\in L^p(0,T; V_p(\mathbb{R}^3))\cap L^\infty (0,T;L^2(\mathbb{R}^3))$ to problem (3.7), then by interpolation inequality, one follows that $u\in L^{\frac{5p}{3}}((0,T)\times \mathbb{R}^3)$. If $p>\frac{9}{5}$ then $p'<\frac{5p}{3}$, and $V^*\otimes u\in L^{p'}((0,t)\times B_a),$ since $V^*\in L^\infty.$
Along the proof of Theorem 2.6 in\cite{wo}, we know that$$ \|\pi\|_{L^{p'}((0,t)\times B_a)}\le C(\|u\|_{L^{p'}((0,t)\times B_a)}+\|S+f^*\|_{L^{p'}((0,t)\times B_a)}+1)$$ where  $B_a$ is any ball of $\RR ^3$ with radius $a$, $C$ only depends on $p,T,f, u_0,a.$

For any $\rho>0$ such that $0<\rho<\delta(\Omega', \Omega)$, Set $\Omega_\rho=\{x\in \Omega;{\rm dist}(x,\Omega')<\rho\}$. Now fix $r<\frac{1}{4}\delta(\Omega',\Omega)$, there exists a ball $B_a$, such that $\Omega\subset B_a$
Let us choose a  cut-off function $\eta$ such that $\eta\equiv1$ on $\Omega_r$, $\eta\equiv 0$ in $\mathbb{R}^3\setminus \Omega_{2r}$, $0\le\eta\le1$ and $|\nabla\eta|<\frac{C}{r},\,\, |\nabla^2\eta|<\frac{C}{r^2}$ in $\Omega_{2r}$, where the constant $C$ depends only on the geometry of $\p\Omega$.

If $|\lambda|<r$, it results that $\Delta_{-\lambda}(\eta^2\Delta_\lambda u)\in L^2((0,T)\times\Omega_{3r})\cap L^p(0,T; W_0^{1,p}(\Omega_{3r})),$ but it is not divergence free. Take $\Delta_{-\lambda}(\eta^2\Delta_\lambda u)$ as a test function in the  first equation  of (3.7), we obtain
\begin{equation*}
\begin{aligned}
&\langle u_t,\Delta_{-\lambda}(\eta^2\Delta_\lambda u)\rangle +(S(D(u)),D(\Delta_{-\lambda}(\eta^2\Delta_\lambda u)))+(V^*\cdot\nabla u,\Delta_{-\lambda}(\eta^2\Delta_\lambda u))\\
&=(\pi,\di(\Delta_{-\lambda}(\eta^2\Delta_\lambda u)))+(f,\Delta_{-\lambda}(\eta^2\Delta_\lambda u)).
\end{aligned}
\end{equation*}
Set
\begin{equation*}
\begin{aligned}
&J_1=\langle u_t,\Delta_{-\lambda}(\eta^2\Delta_\lambda u)\rangle, \\
&J_2=(S(D(u)),D(\Delta_{-\lambda}(\eta^2\Delta_\lambda u))),\\
&J_3=(V^*\cdot\nabla u,\Delta_{-\lambda}(\eta^2\Delta_\lambda u)),\\
&J_4=(\pi,\di(\Delta_{-\lambda}(\eta^2\Delta_\lambda u))),\\
&J_5=(f,\Delta_{-\lambda}(\eta^2\Delta_\lambda u)).
\end{aligned}
\end{equation*}
Clearly, $J_1=\frac{\dif}{\dif t}\|\eta\Delta_\lambda u\|^2_{L^2(\Omega_{2r})}$. Let $I_\lambda(u)=\int_{\Omega_{2r}}(1+|D(u)(x+\lambda e_k)|+|D(u)|)^{p-2}|\eta\Delta_\lambda D(u)|^2\dif x.$

\begin{equation*}
\begin{aligned}
&J_2=2\int_{\Omega_{3r}}S(D(u))\Delta_{-\lambda}sym(\Delta_{\lambda}u\otimes\eta\nabla\eta)\dif x-\int_{\Omega_{2r}}\eta^2(\Delta_\lambda(S(D(u)))\Delta_\lambda D(u))\dif x\\
&:=J_{21}-J_{22}.
\end{aligned}
\end{equation*}
Since (2.3),
\begin{equation*}
\begin{aligned}
 &J_{22}\ge C_1\gamma_1\int_{\Omega_{2r}}(1+|D(u)(x+\lambda e_k)|+|D(u)|)^{p-2}|\eta\Delta_\lambda D(u)|^2\dif x=C_1\gamma_1 I_\lambda(u),\\
 &|J_{21}|\le c\|S(D(u))\|_{L^{p'}(\Omega_{3r})}\|\Delta_{-\lambda}sym(\Delta_{\lambda}u\otimes\eta\nabla\eta)\|_{L^p(\Omega_{3r})},
 \end{aligned}
\end{equation*}
Thus
\begin{equation*}
\begin{aligned}
&\|\Delta_{-\lambda}sym(\Delta_{\lambda}u\otimes\eta\nabla\eta)\|_{L^p(\Omega_{3r})}\le |\lambda|\|\nabla sym(\Delta_{\lambda}u\otimes\eta\nabla\eta)\|_{L^p(\Omega_{3r})},\\
&\le |\lambda|\left[\| \nabla\eta|sym(\Delta_{\lambda}u\otimes\nabla\eta)|\|_{L^p(\Omega_{3r})}+\|\eta sym(\Delta_{\lambda}\nabla u_{x_k}\otimes\nabla\eta)\|_{L^p(\Omega_{3r})}\right.\\
&\left.+\|sym(\Delta_\lambda u\otimes\nabla\eta_{x_k})|\|_{L^p(\Omega_{3r})}\right]\\
&\le 4C\frac{\lambda^2}{r^2}\|\nabla u\|_{L^p(\Omega_{3r})}+\frac{2C|\lambda|}{r}\left(\int_{\Omega_{2r}}|\eta\nabla(\Delta_\lambda u)|^p\dif x\right)^{\frac{1}{p}}.
 \end{aligned}
\end{equation*}
Since $\eta\nabla(\Delta_\lambda u)=\nabla (\eta\Delta_\lambda u)-(\nabla\eta)\cdot\Delta_\lambda u$, thus
$$\left(\int_{\Omega_{2r}}|\eta\nabla(\Delta_\lambda u)|^p\dif x\right)^{\frac{1}{p}}\le \frac{C|\lambda|}{r}\|\nabla u\|_{L^p(\Omega_{3r})}+C_p\left(\int_{\Omega_{2r}}|\eta D(\Delta_\lambda u)|^p\dif x\right)^{\frac{1}{p}}$$
However, by  H\"{o}lder's inequality we have
$$\left(\int_{\Omega_{2r}}|\eta D(\Delta_\lambda u)|^p\dif x\right)^{\frac{1}{p}}\le I_\lambda(u)^{\frac{1}{2}}\left(\int_{\Omega_{2r}}(1+|D(u)(x+\lambda e_k)|+|D(u)|)^{p}\dif x\right)^{\frac{2-p}{2}}.$$
Hence
\begin{equation*}
\begin{aligned}
&J_2\le \|S(D(u))\|_{L^{p'}(\Omega_{3r})}\left(\frac{C\lambda^2}{r^2}\|\nabla u\|_{L^{p}(\Omega_{3r})}+CI_\lambda(u)^{\frac{1}{2}}(1+\|\nabla u\|^{\frac{2-p}{2}}_{L^{p}(\Omega_{3r})})\right)-C\gamma_1 I_\lambda(u)\\
&\le (1+\|\nabla u\|^{p-1}_{L^{p}(\Omega_{3r})})\left(\frac{C\lambda^2}{r^2}\|\nabla u\|_{L^{p}(\Omega_{3r})}+CI_\lambda(u)^{\frac{1}{2}}(1+\|\nabla u\|^{\frac{2-p}{2}}_{L^{p}(\Omega_{3r})})\right)-C\gamma_1 I_\lambda(u)\\
&\le \lambda^2 C(|\Omega_{3r}|,r,\epsilon)\left(1+\|\nabla u\|^{p}_{L^{p}(\Omega_{3r})}\right)-(C_1\gamma_1-\epsilon)I_\lambda(u).\\
&|J_5|=|(f,\Delta_{-\lambda}(\eta^2\Delta_\lambda u))|\le \|f\|_{L^{p'}(\Omega_{3r})}|\lambda|\|\nabla(\eta^2\Delta_\lambda u\|_{L^{p}(\Omega_{2r})}\\
&\le C(r,p)\lambda^2\|f\|_{L^{p'}(\Omega_{3r})}\|\nabla u\|_{L^{p}(\Omega_{3r})}+|\lambda|\|\eta^2\Delta_\lambda \nabla u\|_{L^{p}(\Omega_{2r})}\\
&\le |\lambda|I_\lambda(u)^{\frac{1}{2}}(1+\|\nabla u\|^{\frac{2-p}{2}}_{L^{p}(\Omega_{3r})})\|f\|_{L^{p'}(\Omega_{3r})}+C(\frac{1}{r^2}+\frac{1}{r})\lambda^2\|f\|_{L^{p'}(\Omega_{3r})}\|\nabla u\|_{L^{p}(\Omega_{3r})}\\
&\le C(p,|\Omega_{3r}|)\lambda^2\left(\|f\|^{p'}_{L^{p'}(\Omega_{3r})}+\|\nabla u\|^p_{L^{p}(\Omega_{3r})}+1\right)+\epsilon I_\lambda(u).
\end{aligned}
\end{equation*}
From the  estimate for pressure, divergence-free and the method above, we have
\begin{equation*}
\begin{aligned}
&|J_4|\le 2\left|\int_{\Omega_{3r}}\pi\Delta_\lambda(\eta\eta_i\Delta_\lambda u_i)\dif x\right|\\
&\le C|\lambda|\|\pi\|_{L^{p'}(\Omega_{3r})}\|\frac{\p}{\p x_k}(\eta\eta_{x_i}\Delta_\lambda u_i)\|_{L^{p}(\Omega_{3r})}\\
&\le C\frac{1}{r^2}|\lambda|^2\|\pi\|_{L^{p'}(\Omega_{3r})}\|\nabla u\|_{L^{p}(\Omega_{3r})}+\frac{C|\lambda|}{r}|\lambda|^2\|\pi\|_{L^{p'}(\Omega_{3r})}\|\eta\nabla\Delta_\lambda u\|_{L^{p}(\Omega_{2r})}\\
&\le C\frac{1}{r^2}|\lambda|^2\left(\|\pi\|^{p'}_{L^{p'}(\Omega_{3r})}+\|\nabla u\|^p_{L^{p}(\Omega_{3r})}+1\right)+\epsilon I_{\lambda}(u).
\end{aligned}
\end{equation*}
Now we estimate the term $J_3$. In fact,
\begin{equation*}
\begin{aligned}
&J_3=\int_{\Omega_{3r}}\eta^2\Delta_\lambda V^*\cdot\nabla u \Delta_\lambda u\dif x-2\int_{\Omega_{3r}}\eta V^*\cdot\nabla\eta |\Delta_\lambda u|^2\dif x\\
&:=J_{31}+J_{32}
\end{aligned}
\end{equation*}
Since $\frac{9}{5}<p\le2$, then $2<q=\frac{6p}{5p-6}<\frac{3p}{3-p}=p^*$, by  H\"{o}lder inequality, we have
\begin{equation*}
\begin{aligned}
&J_{31}=\|\eta\Delta_\lambda V^*\|_{L^6(\Omega_{3r})}\|\nabla u\|_{L^{p}(\Omega_{3r})} \|\eta\Delta_\lambda u\|_{L^q(\Omega_{3r})}\\
&\le C|\lambda|\|\nabla u\|_{L^{p}(\Omega_{3r})}\|\nabla V^*\|_{L^6(\Omega_{3r})} \|\eta\Delta_\lambda u\|_{L^q(\Omega_{3r})}\\
&\le C|\lambda|\|\nabla u\|_{L^{p}(\Omega_{3r})}\|V^*\|_{W^{2,2}(\Omega_{3r})} \|\eta\Delta_\lambda u\|_{L^q(\Omega_{3r})}\\
&\le C|\lambda|\|V\|_{L^{\infty}((0,T)\times\Omega)}\|\nabla u\|_{L^{p}(\Omega_{3r})}\|\eta\Delta_\lambda u\|_{L^q(\Omega_{3r})}\\
&J_{32}\le\frac{C}{r}\|V^*\|_{L^6(\Omega_{3r})}\|\eta\Delta_\lambda u\|_{L^q(\Omega_{3r})}\|\Delta_\lambda u\|_{L^{p}(\Omega_{3r})}\\
&\le C|\lambda|\|V\|_{L^{\infty}((0,T)\times\Omega)}\|\nabla u\|_{L^{p}(\Omega_{3r})}\|\eta\Delta_\lambda u\|_{L^q(\Omega_{3r})}.
\end{aligned}
\end{equation*}
It is implies
$$|J_3|\le C(\frac{1}{r},|\Omega_{3r}|)|\lambda|\|V\|_{L^{\infty}((0,T)\times\Omega)}\|\nabla u\|_{L^{p}(\Omega_{3r})}\|\eta\Delta_\lambda u\|_{L^q(\Omega_{3r})}$$
Since $2<q<p^*$, from the following interpolation inequalities
\begin{eqnarray*}
\|u\|_{L^q}\le \|u\|_{L^{p^*}}^\theta\|u\|^{1-\theta}_{L^2};\\
\|u\|_{L^q}\le \|u\|_{L^{p^*}}^{\theta_1}\|u\|^{1-\theta_1}_{L^p},
\end{eqnarray*}
where $\theta=\frac{6-2p}{5p-6},\theta_1=\frac{12-5p}{2p}$. We obtain
\begin{equation*}
\begin{aligned}
&|J_3|&\le C(\frac{1}{r},|\Omega_{3r}|)|\lambda|\|V\|_{L^{\infty}((0,T)\times\Omega)}\|\nabla u\|_{L^{p}(\Omega_{3r})}\|\eta\Delta_\lambda u\|^{(1-\alpha)(1-\theta_1)}_{L^p(\Omega_{3r})}
\\
&&\|\eta\Delta_\lambda u\|^{(1-\theta)\alpha}_{L^2(\Omega_{2r})}\|\eta\Delta_\lambda u\|^{\alpha\theta+(1-\alpha\theta_1)}_{L^{p^*}(\Omega_{3r})}\\
&\le& C\|V\|_{L^{\infty}((0,T)\times\Omega)}|\lambda|^{1+Q_3}\|\nabla u\|^{1+Q_3}_{L^{p}(\Omega_{3r})}\|\eta\Delta_\lambda u\|^{Q_1}_{L^2(\Omega_{2r})}\|\eta\Delta_\lambda u\|^{Q_2}_{L^{p^*}(\Omega_{2r})}
\end{aligned}
\end{equation*}
where
\begin{eqnarray*}
&&Q_1=(1-\theta)\alpha=\frac{7p-12}{5p-6}\alpha,\,\,0<\alpha<1,\\
&&Q_2=\alpha\theta+(1-\alpha\theta_1)=\frac{6-2p}{5p-6}\alpha+\frac{12-5p}{2p}(1-\alpha)\\
&&Q_3=(1-\alpha)(1-\theta_1)=\frac{7p-12}{2p}(1-\alpha).
\end{eqnarray*}
Since $$\|\eta\Delta_\lambda u\|_{L^{p^*}(\Omega_{2r})}\le C\|\nabla(\eta\Delta_\lambda u)\|_{L^{p}(\Omega_{2r})}\le C\|D(\eta\Delta_\lambda u)\|_{L^{p}(\Omega_{2r})}\le C I_\lambda(u)^{\frac{1}{2}}(1+|\nabla u|^{\frac{2-p}{2}}_{L^{p}(\Omega_{2r})}).$$ It infers that
\begin{equation*}
\begin{aligned}
&|J_3|\le C\|V\|_{L^{\infty}((0,T)\times\Omega)}|\lambda|^{1+Q_3}\|\nabla u\|^{1+Q_3}_{L^{p}(\Omega_{3r})}\|\eta\Delta_\lambda u\|^{Q_1}_{L^2(\Omega_{2r})}I_\lambda(u)^{\frac{1}{2}Q_2}\left(1+|\nabla u|^{\frac{2-p}{2}Q_2}_{L^{p}(\Omega_{2r})}\right)\\
&\le C\|V\|_{L^{\infty}((0,T)\times\Omega)}|\lambda|^{1+Q_3}\left(1+\|\nabla u\|^{1+Q_3+\frac{2-p}{2}Q_2}_{L^{p}(\Omega_{3r})}\right)\|\eta\Delta_\lambda u\|^{Q_1}_{L^2(\Omega_{2r})}I_\lambda(u)^{\frac{1}{2}Q_2}
\end{aligned}
\end{equation*}
By Young's Inequality, we get 
$$|J_3|\le C\|V\|_{L^{\infty}((0,T)\times\Omega)}|\lambda|^{(1+Q_3)\delta'}\left(1+\|\nabla u\|^{(1+Q_3+\frac{2-p}{2}Q_2)\delta'}_{L^{p}(\Omega_{3r})}\right)\|\eta\Delta_\lambda u\|^{Q_1\delta'}_{L^2(\Omega_{2r})}+\epsilon I_\lambda(u).$$
Then choose $\delta$ and $\delta'$ satisfy the following identities
\begin{eqnarray*}
\frac{Q_2}{2}\delta=1,\,\, (1+Q_3+\frac{2-p}{2}Q_2)\delta'=p,\,\, \mbox{and}\,\,\frac{1}{\delta}+\frac{1}{\delta'}=1.
\end{eqnarray*}
From these identities, we can obtain $\alpha=\frac{(5p-6)(2-p)}{7p-12}$.  Since $p>\frac{9}{5}$, thus $0<\alpha<1$, $(1+Q_1+Q_3)\delta'=2$ and $\frac{Q_1\delta'}{2}<1$. Therefore,
$$|J_3|\le C\|V\|_{L^{\infty}((0,T)\times\Omega)}|\lambda|^{(1+Q_3)\delta'}\left(1+\|\nabla u\|^{p}_{L^{p}(\Omega_{3r})}\right)\|\eta\Delta_\lambda u\|^{Q_1\delta'}_{L^2(\Omega_{2r})}+\epsilon I_\lambda(u).$$
Combined these relations of $J_1,\cdots,J_5$, we can conclude that
\begin{equation}\label{3.8}\tag{3.8}
\begin{aligned}
&\frac{\dif}{\dif t}\|\eta\Delta_\lambda u\|^2_{L^2(\Omega_{2r})}+(C_1\gamma_1-4\epsilon)I_\lambda(u)\\
&\le C\lambda^2\left(1+\|\nabla u\|^{p}_{L^{p}(\Omega_{3r})}+\|\pi\|^{p'}_{L^{p'}(\Omega_{3r})}+\|f\|^{p'}_{L^{p'}(\Omega_{3r})} \right)\\
&+C\|V\|_{L^{\infty}((0,T)\times\Omega)}|\lambda|^{(1+Q_3)\delta'}\left(1+\|\nabla u\|^{p}_{L^{p}(\Omega_{3r})}\right)\|\eta\Delta_\lambda
u\|^{Q_1\delta'}_{L^2(\Omega_{2r})}
\end{aligned}
\end{equation}
Since $(1+Q_1+Q_3)\delta'=2$ and set
\begin{equation}\label{3.9}\tag{3.9}
\begin{aligned}
&r(t)=C\|V\|_{L^{\infty}((0,T)\times\Omega)}\left(1+\|\nabla u\|^{p}_{L^{p}(\Omega_{3r})}\right),\\
&h(t)=C \left(1+\|\nabla u\|^{p}_{L^{p}(\Omega_{3r})}+\|\pi\|^{p'}_{L^{p'}(\Omega_{3r})}+\|f\|^{p'}_{L^{p'}(\Omega_{3r})} \right).
\end{aligned}
\end{equation}
where $C$ does not depend on $\lambda, u.$
Hence (3.8) can be rewritten
\begin{equation*}
\begin{aligned}
&\frac{\dif}{\dif t}\|\frac{\eta\Delta_\lambda u}{\lambda}\|^2_{L^2(\Omega_{2r})}+(C_1\gamma_1-4\epsilon)\frac{I_\lambda(u)}{\lambda^2}\\
&\le  h(t)+ r(t)\|\frac{\eta\Delta_\lambda u}{\lambda}\|^{Q_1\delta'}_{L^2(\Omega_{2r})}.
\end{aligned}
\end{equation*}
 Since $p>\frac{9}{5}$,it is easy to know that $\frac{Q_1\delta'}{2}<1$. From Young's inequality, we have
 \begin{equation*}
\begin{aligned}
&\frac{\dif}{\dif t}\|\frac{\eta\Delta_\lambda u}{\lambda}\|^2_{L^2(\Omega_{2r})}+(C_1\gamma_1-4\epsilon)\frac{I_\lambda(u)}{\lambda^2}\\
&\le  (h(t)+ r(t))+r(t)\|\frac{\eta\Delta_\lambda u}{\lambda}\|^2_{L^2(\Omega_{2r})}.
\end{aligned}
\end{equation*}
 By  Gronwall's inequality, it follows
\begin{equation*}
\begin{aligned}
&\|\frac{\eta\Delta_\lambda u}{\lambda}\|^2_{L^2(\Omega_{2r})}(t)+(C_1\gamma_1-4\epsilon)\int_0^t\frac{I_\lambda(u)}{\lambda^2}\\
&\le\left[ \|\frac{\eta\Delta_\lambda u}{\lambda}\|^2_{L^2(\Omega_{2r})}(0)+\int_0^t h(s)\exp{\left(-\int_0^s r(\tau)\dif\tau\right)}\dif s\right]\exp{\int_0^t r(s)\dif s}.
\end{aligned}
\end{equation*}
Assume that $u_0\in V_2$, then it implies that for all $t\in [0,T]$
\begin{equation}\label{3.10}\tag{3.10}
\begin{aligned}
&\|\frac{\eta\Delta_\lambda u}{\lambda}\|^2_{L^2(\Omega_{2r})}(t)+(C_1\gamma_1-4\epsilon)\int_0^t\frac{I_\lambda(u)}{\lambda^2}\\
&\le\left[ \|\nabla u^*_0\|^2_{L^2(\mathbb{R}^3_+)}+\int_0^1(h(s)+r(s))\exp{\left(-\int_0^s r(\tau)\dif\tau\right)}\dif s\right]\exp{\int_0^t r(s)\dif s}.
\end{aligned}
\end{equation}

Choose $\epsilon=\frac{C_1\gamma_1}{8}$ and from (3.9),(3.10), we conclude for any $t\in [0,T]$
\begin{equation}\label{3.11}\tag{3.11}
\begin{aligned}
\|\nabla u\|^2_{L^2(\Omega_{r})}(t)+\int_0^tI(u)\le C.
\end{aligned}
\end{equation}
where $I(u)=\int_{\Omega_r}(1+2|D(u)|)^{p-2}|D(\nabla u)|^2\dif x$ and $C$ depends on $\|V\|_{L^{\infty}((0,T)\times\Omega)},p,a,u_0,f,T,r.$ Since $$\int_0^t\|\nabla^2 u\|^p_{L^p(\Omega_r)}\dif x\le \int_0^tI(u)+\int_0^t\|\nabla u\|^p_{L^p(\Omega_r)}$$ for all $t\in [0,T]$, thus $\|u\|_{L^p(0,T;W^{2,p}(\Omega_r))}\le C.$
Multiply  the first equation of (3.7) by $\eta^2 u_t$ and integrate on the $\Omega_{3r}$, one obtains
\begin{equation}\label{3.12}\tag{3.12}
\begin{aligned}
&\|\eta u_t\|^2_{L^2(\Omega_{3r})}(t)+\frac{\dif}{\dif t} \|\eta^2\Phi(D(u))\|_{L^1(\Omega_{3r})}\\
&\le 2\int_{\Omega_{3r}}|\nabla\eta S(D(u))\eta u_t|\dif x+\int_{\Omega_{3r}}V^*\cdot\nabla u\eta^2u_t\dif x\\
&+2\int_{\Omega_{3r}}\pi\nabla\eta\cdot\eta u_t\dif x+\int_{\Omega_{3r}}f\cdot\eta^2 u_t\dif x.
\end{aligned}
\end{equation}
Since $\frac{1}{2}+\frac{1}{p'}<1$, by H\"{o}lder inequality, one can
\begin{equation*}
\begin{aligned}
&\int_{\Omega_{3r}}|\nabla\eta S(D(u))\eta u_t|\dif x\le\|\nabla\eta\|_{L^q(\Omega_{3r})}\|S(D(u))\|_{L^{p'}(\Omega_{3r})}\|\eta u_t\|_{L^2(\Omega_{3r})}\\
&\int_{\Omega_{3r}}V^*\cdot\nabla u\eta^2u_t\dif x\le C \|V\|_{L^{\infty}((0,T)\times\Omega)}\|\eta\nabla u\|_{L^2(\Omega_{2r})}\|\eta u_t\|_{L^2(\Omega_{3r})}\\
&\int_{\Omega_{3r}}\pi\nabla\eta\cdot\eta u_t\dif x+\int_{\Omega_{3r}}f\cdot\eta^2 u_t\dif x\le C\|\nabla\eta+\eta\|_{L^q(\Omega_{3r})}\|\pi+f\|_{L^{p'}(\Omega_{3r})}\|\eta u_t\|_{L^2(\Omega_{3r})}.
\end{aligned}
\end{equation*}
Thanks to Young's inequality, we can obtain that
\begin{equation*}
\|\eta u_t\|^2_{L^2(\Omega_{3r}\times (0,t))}+ \|\eta^2\Phi(D(u))\|_{L^\infty(0,t;L^1(\Omega_{3r}))}\le K
\end{equation*}
where $K$ depends only on $p,a,u_0,f,r, \|V\|_{L^{\infty}((0,T)\times\Omega)}, \|\nabla u\|_{L^p(\mathbf{R}^3\times (0,T))}.$

In fact, from this proof, we can see that the bound depends on the measure of $\Omega$ and $\Omega'$. Hence,  if the radius of the  ball $B$ is fixed, then $\|\nabla^2u\|_{L^p(0,T;L^p{B})}\le C$,  consequently,  $u\in C^\gamma(B)$. We use the following argument ( see\cite{ga}) to know that  $u(x,t)\to 0$ for almost$t\in(0,T)$, as $|x|\to \infty.$ Let the radius $B$ be one, suppose that there exist $\epsilon>0$ and a sequence $\{x_n\}\subset \mathbb{R}^3$ with $\lim_{n\to\infty}|x_n|\to\infty$, such that for almost $t\in(0,T)$, $u(x_n,t)\ge\epsilon.$  By the continuity of $u(x,t)$, then we get that if $|x-x_n|\le \delta=\min\{1,(\frac{\epsilon}{2C})^{\frac{1}{\gamma}}\}$ then $u(x,t)\ge \frac{\epsilon}{2}$. Without loss of generality, we assume that $|x_i-x_j|>2$ provided that $i\neq j$ thus
\begin{equation*}
\int_0^T(\int_{\mathbb{R}^3} |u|^{p^*}\dif x)^\frac{p}{p^*}\dif t\ge\sum_{j}\int_0^T(\int_{B(x_j)} |u|^{p^*}\dif x)^\frac{p}{p^*}\dif t =+\infty.
\end{equation*}
and this contradict with  the fact $u\in L^p(0,T;V_p(\mathbb{R}^3)).$

From the fact $u(x,t)\to 0$ a.e. $t\in (0,T)$,  as $|x|\to\infty$ and estimate (3.13), we can conclude that the weak solution is unique. Actually, assume that $u,v$ are that weak solutions of problem (3.8), then set $e=u-v$ and multiply the difference of the equations of $u$ and $v$ by $e$. After integrate over  $\mathbb{R}^3$, we have
$$(\p_t e,e)+(S(D(u))-S(D(v)),D(u)-D(v))=0.$$
It reduces to
\begin{equation}\label{3.15}\tag{3.13}
\frac{\dif}{\dif t}\|e\|^2_{L^2(\mathbb{R}^3)}+(S(D(u))-S(D(v)),D(u)-D(v))=0\end{equation}
Use  Gronwall's inequality in (3.14), we have $e=0$, i.e. $u=v.$  Define
 \begin{equation*}
 \begin{aligned}
 &u^*=(u_1(x_1,x_2,-x_3),u_2(x_1,x_2,-x_3),-u_3(x_1,x_2,-x_3)\\
 &\pi^*=\pi(x_1,x_2,-x_3),\,\, \forall x\in\mathbb{R}^3.
 \end{aligned}
 \end{equation*}
 Then by the method in \cite{B3}, the couple $(u^*,\pi^*)$ is also a solution to problem (3.4) a.e. in $\mathbb{R}^3_+\times (0,T).$ Hence by the uniqueness, we know that $u^*=u$.

 From the regularity of $u$, and Theorem 7.1 of \cite{B3} and Sobolev imbedding theorems, we know the solution $u(t)$ to problem (1.1) is simply the restriction of $u^*(t)$ to the half-space $\mathbb{R}^3_+$, and $\|\nabla u^*\|_{L^p(\mathbb{R}^3\times (0,T))}\le C\|\nabla u\|_{L^p(\mathbb{R}^3_+\times (0,T))},$  $\|u^*\|_{L^\infty(0,T;L^2(\mathbb{R}^3))}\le C\|u\|_{L^\infty(0,T;H)},$ and  $u(x)$ is a continuous function on $B_r$ for any $r>0$. So that from $u_3(x)=-u_3(x_1,x_2,-x_3)$ we know that $u_3|_{x_3=0}=0.$ Analogously $u_i(x)=u_i(x_1,x_2,-x_3),\,\forall x\in\mathbb{R}^3(i=1,2)$, satisfy the conditions (1.5) in the sense of trace.

 The theorem is completely proved.
\end{proof}

By the minor modification of the proof in theorem above, we can obtain the regularity results in the case $p\ge 2$.
\begin{Lemma}
Let $p\ge 2,f\in L^{p'}(Q_T)\cap L^2(Q_T)$, $u_0\in V_2\cap H$ satisfies the boundary conditions (1.6) , and $S$ given by a p-potential from Definition 1. If  $u\in L^p(0,T; V_p)\cap L^\infty (0,T;H)$ is the weak solution for problem (1.1), then this solution is also a unique strong solution to problem (1.1) such that
\begin{alignat*}{12}
&\|u\|_{L^{\infty}(0,T; W^{1,2}(\Omega'))\cap L^2(0,T; W^{2,2}(\Omega'))}\le C(\delta(\Omega',\Omega),|\Omega'|,u_0,f,T,\|V\|_{L^\infty((0,T)\times\Omega)}),\\
&\int^T_0\int_{\Omega'}(1+|D(u)|)^{p-2}|\nabla D(u)|^2\dif x\dif t\le C(\delta(\Omega',\Omega),|\Omega'|,u_0,f,T,\|V\|_{L^\infty((0,T)\times\Omega)}).\\
&\|\frac{\p u}{\p t}\|^2_{L^2((0,T)\times\Omega')}+\|\Phi(D(u))\|_{L^\infty(0,T;L^1(\Omega'))}\le C(\delta(\Omega',\Omega),|\Omega'|,u_0,f,T,\|V\|_{L^\infty((0,T)\times\Omega)}).\\
&\|\pi\|_{L^2((0,T)\times\Omega')}\le C(\delta(\Omega',\Omega),|\Omega'|,u_0,f,T,\|V\|_{L^\infty((0,T)\times\Omega)}).
\end{alignat*}
\end{Lemma}

\begin{Remark}
Since $\di V=0$, thus if $u$ is a weak solution of problem (1.1),  then we can obtain the following estimates
\begin{equation}\label{3.14}\tag{3.14}
\|u\|^2_{L^\infty (0,T;H)}+\|u\|^p_{L^p(0,T; V_p)}\leq C (\|u_0\|^2_H+\|f\|^{p'}_{L^{p'}(\mathbb{R}^3_+)\times (0,T)}).
\end{equation}
It is easy to see that $C$ does not depend  on $a,r,|\Omega'|, u_0,f,\|V\|_{L^\infty((0,T)\times\Omega)}$.
\end{Remark}

From these regularity estimates presented in Theorem 3.4 and lemma 3.5, we can obtain the existence of unique weak solution to problem (1.1) stated by the following theorem.

\begin{Theorem}
Let $p>\frac{9}{5}, f\in L^{p'}(Q_T)\cap L^2(Q_T),V\in L^\infty((0,T),W^{2,2}(\RR^3_+)), u_0\in V\cap H$, and $S$ given by a p-potential from Definition 2.1. Then there exists a unique weak solution for problem (1.1),   $u \in L^p(0,T; V_p)\cap L^\infty (0,T;H)$ and satisfies the inequality (3.14).
\end{Theorem}
\begin{proof}
  From the proof in Theorem 3.4, it is easy to see that the weak solution is unique. We will use standard Galerkin method to prove its existence.

Let $$\mathbb{R}^3_+=\bigcup_{R=1}^{\infty}\Omega_R=\bigcup_{R=1}^{\infty}\{x\in\mathbb{R}^3_+:|x|\le R\}.$$
Fix $R>0$, we consider the auxiliary problem (2.7)  for the initial $u^R_0=P(\chi_{\Omega_R}(x)u_0(x))$ and external force term $f^R=\chi_{\Omega_R}(x) f$. 

As Lemma 3.5 in \cite{mar}, we construct a sequence of solenoidal vector  functions $V_R$ defined in $\Omega_R$ such that $V_R\to V$ in $L^p (0,T;V_p(\mathbb{R}_+^3))$  as $R\to \infty$,  $\|V_R\|_{L^\infty(0,T;W^{2,2}(\Omega_R))}\le C $, here $C$ depends only on the Geometry of $\Omega_R$, but does not depend on the measure of $\Omega_R$.  We also know that $u^R_0\to u_0$ in $V_2\cap H$ and$f^R\to f$ in $L^{p'}(Q)$, as $R\to\infty$.

Choose the sequence $\{a_k^R\}$ is the eigenvector of the operator $A$ as in Proposition 2.8, then $\{a_k^R\}$ is a basis $W^{2,2}(\Omega_R)\cap V_2(\Omega_R)$. We look for the weak solution to the following problem 
 \begin{equation*}
\left\{\begin{aligned} &\partial_t u-\di S(D(u))+(u\cdot\nabla)u+\nabla \pi=f, &&\mbox{in}~~\Omega_R\times (0, T) \\[3mm]
&\nabla\cdot u = 0, &&\mbox{in}~~\Omega_R\times (0, T), \\[2mm]
&u_3=0,~~~\frac{\p u_1}{\p x_3}=\frac{\p u_2}{\p x_3}=0  &&\mbox{on}~~\Gamma_R^1\times(0, T),\\[2mm]\\
&u=0    &&\mbox{on}~~\Gamma_R^2\times(0, T),\\
&u|_{t=0}=u_0(x),&& \mbox{in}~~\Omega_R.
\end{aligned}\right.
\end{equation*}

  We find the approximation solutions with the form $$u_m^R(x,t)=\sum_{k=1}^m c_{k,m}^R(t) a_k^R(x).$$ For simplicity, in the clear meaning setting, we omit the superscript $R$. Therefore, $c_{k,m}(t)$ solve the following system of ordinary differential equations
\begin{equation}\label{3.15}\tag{3.15}
\begin{aligned}
\frac{\dif }{\dif t}(u_m,a_k)+(S(D(u_m)),D (a_k))-(V\otimes u_m,\nabla a_k)=(f,a_k).
\end{aligned}
\end{equation}
Due to the continuity of $S, V$, the local-in-time existence  follows from Caratheodory theory. The global-in-time existence will be established by the following a-priori estimates.

Multiply the equations (3.15) by $c_{k,m}$,  then sum over $k$ and integrate on $(0,t)$. We easily obtain
\begin{equation}\label{3.16}\tag{3.16}
\begin{aligned}
\|u_m\|^2_{L^2(\Omega_R)}+\int_0^t(S(D(u_m)),D (u_m))=(f,u_m)+\|u_0\|_{L^2(\Omega_R)}.
\end{aligned}
\end{equation}
Hence
\begin{equation}\label{3.17}\tag{3.17}
\begin{aligned}
\sup_{0\le t\le T}\|u_m\|^2_{L^2(\Omega_R)}(t)+\|u_m\|^p_{L^p(0,T;V_p(\Omega_R))}\le  \|u_0\|^2_H+\|f\|^{p'}_{L^{p'}(\mathbb{R}^3)\times(0,T)}:=M.
\end{aligned}
\end{equation}

From (3.15) and (3.17),  we infer that $\|u_m'\|_{L^{p'}(0,T; (V_p(\Omega_R))^*)}\le C(R,M).$

From these estimates, and Aubin-Lions Lemma, we have
\begin{alignat*}{12}
&u'_m\rightharpoonup u',\,\, &&~\mbox{weakly\,\, in}~ L^{p'}(0,T;V_p(\Omega_R)^*)\tag{3.18};\\
&u_m \stackrel{*}{\rightharpoonup} u,\,\, &&~\mbox{weakly}^* \mbox{\,\, in}~ L^{\infty}(0,T;L^2(\Omega_R))\tag{3.19};\\
&u_m \rightharpoonup u,\,\,&&~\mbox{weakly\,\, in}~ L^p(0,T;V_p(\Omega_R)),\tag{3.20};\\
&u_m\longrightarrow u,\,\,&&~\mbox{strongly\,\, in}~ L^q(0,T;L^q(\Omega_R))\,~ q\in [1,\frac{5p}{3})\tag{3.21};\\
&S(D(u^R_m))\rightharpoonup \tilde{S}^R,\,\,&&~\mbox{weakly\,\, in}~ L^{p'}(\Omega_R\times(0,T)) \tag{3.22}.
\end{alignat*}
It is easy to see that
\begin{equation}\label{3.23}\tag{3.23}
\begin{aligned}
\inner{u_t}{a_k}+(\tilde{S}^R,D (a_k))-(V\cdot\nabla a_k,u)=(f,a_k).
\end{aligned}
\end{equation}
Multiplying both sides of (3.23) by $c_{k,m}$ and summing over $k$ we find
\begin{equation}\label{3.24}\tag{3.24}
\begin{aligned}
\inner{u_t}{u_m}+(\tilde{S}^R,D (u_m))-(V\cdot\nabla u_m,u)=(f,u_m).
\end{aligned}
\end{equation}
Let us pass to the limit for $m\to\infty$ in to (3.24).  By the convergence properties (3.18),  (3.20) and (3.22), we know that as $m\to\infty$
\begin{eqnarray*}
&&\int_0^T \inner{u_t}{u_m}\to \int_0^T\inner{u_t}{u}=\|u\|^2_{L^2(\Omega_R)}-\|u_0\|^2_{L^2(\Omega_R)},~\,\\
&&\int_0^T(f,u_m)\to\int_0^T(f,u),~ \int_0^T(\tilde{S}^R,D (u_m))\to\int_0^T(\tilde{S}^R,D (u)).
\end{eqnarray*}
Since
\begin{equation*}
(V\cdot\nabla u_m,u)-(V\cdot\nabla u,u)=\int_{\Omega_R}(u_{\eta(\varepsilon)}\otimes u)\cdot\nabla(u_m-u)\dif x
\end{equation*}
From (3.20) and (3.21), we know that $V\otimes u\in L^{p'}(\Omega_R\times(0,T))$ whenever $p>\frac{9}{5}.$ Hence $\int_0^T(V\cdot\nabla u_m,u)\to 0$ as $m\to \infty$, since $(V\cdot\nabla u,u)=0.$
Subtracting (3.24) by (3.15) and passing to limit as $m\to\infty$, we get
\begin{equation}\label{3.25}\tag{3.25}
\lim_{m\to\infty}\int_0^T(S(D(u_m^R)),D(u_m))\dif t=\int_0^T(\tilde{S^R},D(u))\dif t
\end{equation}
By the monotonicity property (2.4), we can write the following inequality
\begin{equation}\label{3.26}\tag{3.26}
\int_0^T(S(D(u_m^R))-S(D(\Psi)),(D(u_m)-D(\Psi)))\dif t\ge 0,
\end{equation}
For any $\Psi\in L^p(0,T;V_p(\Omega))$. Thus, pass to the limit as $m$ goes to infinity into this relation and using (3.20), (3.22)and (3.25), we have
 \begin{equation}\label{2.27}\tag{2.27}
\int_0^T(\tilde{S}^R-S(D(\Psi)),D(u)-D(\Psi))\dif t\ge 0.
\end{equation}
For all $\Psi\in L^p(0,T;V_p(\Omega))$, Take $\Psi=u-\epsilon \varphi$, $\epsilon>0$ and $\varphi\in L^p(0,T;V_p(\Omega))$, we have that
\begin{equation*}
\int_0^T(\tilde{S}^R-S(D(u-\epsilon \varphi)),D(\varphi))\dif t\ge 0,
\end{equation*}
Letting $\epsilon\to 0$ and using the continuity of $S$, we arrive at
\begin{equation*}
\int_0^T(\tilde{S}^R-S(D(u)),D(\varphi))\dif t\ge 0, \forall \varphi\in L^p(0,T;V_p(\Omega))
\end{equation*}
Choose $-\varphi$ in place of $\varphi$, we get
\begin{equation*}
\int_0^T(\tilde{S}^R-S(D(u)),D(\varphi))\dif t\le 0, \forall \varphi\in L^p(0,T;V_p(\Omega))
\end{equation*}
This implies that $\tilde{S}^R=S(D(u))$ a.e. $\Omega_R\times (0,T).$  Thus the existence of weak solution $u^R$ to problem above is proved.

Next we must consider the limits as $R$ tend to $\infty$. Now we choose a sequence of real number $\{R_N:N\in\mathbf{N}\}$ increasing to infinity. We set $u_N=u^{R_N}$ and extend $u_N$ to zero outside $\Omega_{R_N}$ to obtain a function still denote $u_N\in L^\infty(0,T;H)\cap L^p(0,T; V_p)$ and satisfies the following apriori estimate
\begin{equation}\label{3.28}\tag{3.28}
\begin{aligned}
\sup_{0\le t\le T}\|u_N\|_{H}(t)+\|u_N\|_{L^p(0,T;V_p)}\le  \|u_0\|^2_H+\|f\|^{p'}_{L^{p'}(\mathbb{R}^3)\times(0,T)}:=M.
\end{aligned}
\end{equation}
Clearly, $M$ is independent of $N$. From (3.28) and interpolation inequality we have   $\|V_N\otimes u_N\|_{L^{p'}(\mathbb{R}^3_+\times(0,T))}\le C(M)$.  let $Y=V_p$, Therefore,  from the equations (3.23) we obtain
 \begin{equation*}
 \|u'_N\|_{L^{p'}(0,T;Y^*)}\le \|S(D(u_N))\|_{L^{p'}(Q)}+\|V_N\otimes u_N\|_{L^2(Q)}\le C(M).
 \end{equation*}
By the Aubin-Lions Lemma, we know that there exists a subsequence $u_{N_k}\to u$ in $L^{p}(\Omega_R\times(0,T))$, thus $u_{N_k}\to u$ a.e. in $\mathbb{R}^3_+\times[0,T]$.

We choose $\phi\in C_0^{\infty}(\overline{\mathbb{R}^3_+}\times[0,T))$ with $\di\phi=0,\,\, \phi_3|_{x_3=0}=0$, and $\text{supp} \phi\subset \overline{\mathbb{R}^3_+}\times [0,T).$  There exists a number $K=K(\phi)>0$ such that $G=\text{supp}\phi\subsetneqq \Omega_{N_K}\times [0,T)$ for $k>K$. From the formula (2.6),  we can obtain the following identity
 \begin{equation}
 \begin{aligned}
 &-\int_Q (u_{N_k}\cdot\p_t\phi)\dif x\dif t+\int_QS(x,t,D(u_{N_k})):D(\phi)\dif x\dif t\\
 &-\int_Q (V_{N_k}\otimes u_{N_k}):D(\phi)\dif x\dif t=\int_{\mathbb{R}^3_+}u_0\cdot\phi(0)\dif x.
\end{aligned}\tag{3.29}
 \end{equation}

 From the estimate (3.28), we know that $u_{N_k}\rightharpoonup u$ in $L^2(Q)$, $S(x,t,D(u_{N_k})\rightharpoonup \tilde{S}$ in the space $L^{p'}(Q)$ and $|V_{N_k}\otimes u_{N_k}|_{L^2(Q)}\le C\|u_0\|_H\|V\|_{L^\infty(G)}$. By  Vitali's theorem, we have  as $k\to\infty$
 \begin{equation*}
 \begin{aligned}
 &\int_Q (u_{N_k}\cdot\p_t\phi)\dif x\dif t\to \int_Q (u\cdot\p_t\phi)\dif x\dif t,\,\\
 &\int_Q(V_{N_k}\otimes u_{N_k}):D(\phi)\dif x\dif t\to\int_Q(V_{N_k}\otimes u):D(\phi)\dif x\dif t.
 \end{aligned}
 \end{equation*}
 From these convergence and the formula (3.29), we know that $$\int_Q(S(x,t,D(u_{N_k})):D(\phi)\dif x\dif t\to\int_Q(\tilde {S}:D(\phi)\dif x\dif t.$$

One must check that $S(x,t,D(u))=\tilde {S}$  a.e. in $\mathbb{R}^3_+\times[0,T].$ It suffices to prove that  $$\int_Q(S(x,t,D(u_{N_k})):D(\phi)\dif x\dif t\to\int_Q(S(D(u)):D(\phi)\dif x\dif t.$$ Indeed, observing that for any $R_N>R_{N_K}$ the solutions $u_N$ satisfies the hypotheses of Theorem 3.4 and lemma 3.5 with $\Omega=\Omega_{N_K}$ and fixed a domain  $\Omega'$ such that $G\subset\Omega'\times(0,T)\subset\subset\Omega\times(0,T)$, we get that $u_{N}\in L^p(0,T;W^{2,p}(\Omega'))$ and
\begin{equation*}\begin{aligned}
&\|u'_{N_k}\|_{L^2((0,T)\times\Omega')}+\|\nabla u_{N_k}\|_{L^\infty(0,T; L^2(\Omega'))}\\
&+\|\nabla^2 u_{N_k}\|_{L^p(\Omega'\times(0,T))}\le C(\delta(\Omega',\Omega),|\Omega'|,u_0,f,T,\|V\|_{L^\infty(G)}), ~\mbox{in\ \ the\ \ case\ \ } \frac{9}{5}<p<2;\\
&\|u'_{N_k}\|_{L^2((0,T)\times\Omega')}+\|\nabla u_{N_k}\|_{L^\infty(0,T; L^2(\Omega'))}\\
&+\|\nabla^2 u_{N_k}\|_{L^2(\Omega'\times(0,T))}\le C(\delta(\Omega',\Omega),|\Omega'|,u_0,f,T,\|V\|_{L^\infty(G)}), ~\mbox{in\ \ the\ \ case\, \ } p\ge 2.
\end{aligned}
\end{equation*}
From the boundedness above and the Aubin-Lions lemma we obtain that
$$\nabla u_{N_k}\to \nabla u\,\,~\mbox{in}~ L^p(\Omega'\times(0,T)),~ \,\,\nabla u_{N_k}\to \nabla u~\mbox{a.e. in}~ G.$$
Therefore, $S(x,t,D(u_{N_k}))\to S(x,t,D(u))~\mbox{a.e. in}~ G$, then by  Vitali's theorem, we obtain as $k\to\infty$
\begin{equation*}
\begin{aligned}
&\int_Q\left((S(x,t,D(u_{N_k}))-(S(D(u))\right):D(\phi)\dif x\dif t =\\
&\int_G\left((S(x,t,D(u_{N_k}))-(S(D(u))\right):D(\phi)\dif x\dif t\to0.
\end{aligned}
\end{equation*}
 Whence, this proves the theorem.
\end{proof}

\end{document}